\documentclass[a4paper,12pt,draft]{amsart}

\usepackage[cp1250]{inputenc}
\usepackage{amssymb}
\usepackage{amsfonts}
\usepackage{enumerate}
\usepackage[top=35mm, bottom=45mm, left=35mm, right=35mm]{geometry}

\usepackage[T1]{fontenc}

\newtheorem{tw}{Theorem}
\newtheorem*{twn}{Theorem}
\newtheorem{lem}[tw]{Lemma}

\theoremstyle{definition}
\newtheorem{df}[tw]{Definition}
\newtheorem{ex}[tw]{Example}

\def\C{\mathbb{C}}
\def\O{\mathcal{O}}
\def\D{\mathbb{D}}
\def\phi{\varphi}
\def\la{\lambda}

\def\lf{l}

\usepackage{color}

\begin{document}

\title[Azukawa izometries and biholomorphisms]{A note on relation between
Azukawa izometries at one point and global biholomorphisms}
\author[M. Zaj\k{e}cka]{Ma{\l}gorzata Zaj\k{e}cka}
\address{\textnormal{Jagiellonian University\newline
\indent Faculty of Mathematics and Computer Science\newline
\indent Institute of Mathematics\newline
\indent {\L}ojasiewicza 6\newline
\indent 30-348 Krak\'ow}}

\address{\textnormal{Cracow University of Technology\newline
\indent Faculty of Physics, Mathematics and Computer Science\newline
\indent Institute of Mathematics\newline
\indent Warszawska 24\newline
\indent 31-155 Krak\'ow\newline
$ $}}

\email{malgorzata.zajecka@gmail.com}
\subjclass[2010]{}
\begin{abstract} 
We prove that under certain assumptions holomorphic functions which are Azukawa
isometries at one point are in fact biholomorphisms.
\end{abstract}
\maketitle

\section{Introduction}

The holomorphic contractibility of Carath\'eodory-Reiffen and
Kobayashi-Royden pseudometrics have put much interest in the relation of
global biholomorphicity and Carath\'eodory or Kobayashi isometricity at one
point. While from mentioned property it immediately follows that a biholomorphism must
be a Carath\'eodory and Kobayashi isometry, the oposite statement is obviously
not true in general case. In 1984 Jean-Pierre Vigu\'e proved the following
result.

\begin{twn}[see \cite{Vig}]
Let $\Omega$ be a bounded convex domain in $\C^n$ and let $M$ be a complex
manifold on which a Carath\'eodory-Reiffen pseudodistance is a distance. Suppose
$F:\Omega\to M$ is a holomorphic mapping which is a Carath\'eodory-Reiffen
isometry at a point $p\in\Omega$. Then $F$ is a biholomorphism.
\end{twn}

A few years later Ian Graham proved analogous theorem for a Kobayashi-Royden
isometry.

\begin{twn}[see \cite{Gra}]
Suppose $M$ is a taut complex manifold of dimension $n$. Suppose $\Omega$ is a
strictly convex bounded domain in $\C^n$. Suppose $F:M\to\Omega$ is a
holomorphic mapping which is a Kobayashi-Royden isometry at a point $p\in M$.
Then $F$ is a biholomorphism.
\end{twn}

In this paper we switch our interest to other holomorphically contractible
pseudometric: the Azukawa pseudometric $A_G$. We obtain the following main
result.

\begin{tw}\label{tw1} 
Let $G_1, G_2\subset\C^n$ be domains. Let $a\in G_1$ and let
$F:G_1\to G_2$ be such that:
\begin{enumerate}[\rm (1)]
  \item $F\in\O(G_1,G_2)$;
  \item $G_1$ is taut;
  \item $G_2$ is bounded;
  \item for any $z\in G_1$ we have $g_{G_1}(a,z)=\lf_{G_1}^*(a,z)$;
  \item for any $X\in\C^n$ we have $A_{G_2}(F(a);F'(a)X)=A_{G_1}(a;X)$.
\end{enumerate}
Then $F$ is a biholomorphism.
\end{tw}

Proof of this result is given in the last section of this paper.

The following example shows essential difference between our result and theorems
of Vigue and Graham.

\begin{ex}
Observe that if $G_1$ is a taut balanced pseudoconvex domain, then assumption
(4) in Theorem \ref{tw1} is satisfied (see \cite{JarPfl},
Proposition 4.2.7). However such domains do not need to be convex, thus in
many cases we cannot apply Vigue's nor Graham's theorems with $G_1$ in the role
of $\Omega$. A good example of that situation is a set $G_1:=\{(z,w)\in\C^n:\
|z|<1,\ |w|<1,\ |zw|<\alpha\}$, for fixed $0<\alpha<1$.
\end{ex}

\section{Preliminaries}

In this section we remind the definitions and basic properties of the Green
function and Azukawa pseudometric. For the detailed proofs, more interesting
facts about these objects and their connections to Kobayashi and Carath\'eodory
pseudodistances and pseudometrics see for instance \cite{JarPfl}, \cite{Kli 1},
\cite{Kli 2}, \cite{Azu}, and \cite{Kli 3}.

Let $G$ be a
domain in $\C^n$. To simplify the definitions, for $a\in G$ let
$\exp\mathcal{L}_a$ denote the family of functions $u:G\to[0,1)$ such that $\log
u$ is plurisubharmonic on $G$ and there exists a positive constant $M$ such that
$u(z)\leq M\|z -a\|$, $z\in G$.

\begin{df}
Define
$$
g_G(a,z):=\sup\{u(z):\ u\in\exp\mathcal{L}_a\},\ a\in G,\ z\in G;$$
$$
A_G(a;X):=\sup\{\limsup_{0\neq\la\to 0}\frac{u(a+\la X)}{|\la|}:\
u\in\exp\mathcal{L}_a\},\ a\in G,\ X\in\C^n.$$
The function $g_G$ is called \emph{a pluricomplex Green function with a pole at
point a} and $A_G$ is called \emph{an Azukawa pseudometric}.
\end{df}

Both $g_G$ and $A_G$ are holomorphically contractible, i.e. for a domain
$D\subset\C^m$ and a holomorphic function $F:G\to D$ we have
$$g_D(F(a),F(z))\leq g_G(a,z),\ a,z\in G,$$ and $$A_D(F(a);F'(a)X)\leq
A_G(a;X),\ a\in G,\ X\in\C^n.$$Obviously if $F$ is a biholomorphism, we get
equalities - this property is called biholomorphic invariance. Moreover
$$g_\D(\la',\la'')=m(\la',\la''):=|(\la'-\la'')/(1-\la'\overline{\la''})|,\
\la',\la''\in\D,$$ and $$A_\D(\la;X)=|X|\gamma(\la):=|X|/(1-|\la|^2),\
\la\in\D,\ X\in\C,$$ where by $\D$ we denote a unit disc in $\C$.

One can show that $g_G(a,\cdot)\in\exp\mathcal{L}_a$, $a\in G$, (see
\cite{Kli 2}). Consequently we obtain an equivalent and much more
useful definition of Azukawa pseudometric$$A_G(a;X)=\limsup_{0\neq\la\to
0}\frac{g_G(a,\la X)}{|\la|},\ a\in G,\ X\in\C^n.$$

Finally, let us introduce the following notation. By $\lf_G$ we denote the
Lempert function, by $k_G$ the Kobayashi pseudodistance and by $K_G$ the
Kobayashi-Royden pseudometric. We use also the convention: for a
function $f$ let $f^*$ denote $\tanh f$.

\section{Proof of main theorem}

Before we proceed to the proof, we need one easy but interesting lemma.

\begin{lem}\label{l1}
Let $G\subset\C^n$ be a domain such that $0\in G$ and let $\phi\in\O(\D,G)$ such
that $\phi(0)=0$. Then the following conditions are equivalent:
\begin{enumerate}[\rm (i)]
  \item there exists a $\la'\in\D\setminus\{0\}$
  such that $g_G(0,\phi(\la'))=|\la'|$;
  \item $g_G(0,\phi(\la))=|\la|$ for all $\la\in\D$;
  \item $A_G(0;\phi'(0))=1$.
\end{enumerate}
\end{lem}

\begin{proof}
The proof is similar to the proof of analogous theorem for complex
Carath\'eodory and Carath\'eodory-Reiffen geodesics (see Proposition 8.1.3 in
\cite{JarPfl}).
\end{proof}

\begin{proof}[Proof of Theorem \ref{tw1}]
Using biholomorphic invariance of $g_G$ and $A_G$ we may assume $a=0$ and
$F(a)=0$. Observe that $A_{G_1}(0;X)>0$ for $X\neq 0$. Indeed, using (4),
Proposition 3.18 from \cite{Pan}, and tautness of $G_1$ we obtain
\begin{multline*}
A_{G_1}(0;X)=\limsup_{0\neq\la\to 0}\frac{g_{G_1}(0,\la
X)}{|\la|}=\limsup_{0\neq\la\to 0}\frac{\lf^*_{G_1}(0,\la
X)}{|\la|}\\ =\limsup_{0\neq\la\to 0}\frac{\lf_{G_1}(0,\la X)}{|\la|}
=K_{G_1}(0;X)>0.
\end{multline*}

Now, from (5) we get $A_{G_2}(0;F'(0)X)>0$, $X\neq 0$. Thus $F'(0)$ is an
isomorphism and so $F$ is injective in a neighborhood $U$ of zero. From tautness
of $G_1$ we get equality between the euclidean topology 
of $G_1$ and its Kobayashi topology (see \cite{JarPfl} Proposition 3.3.4). Using (5) we may
assume $U\supset B_{g_{G_1}}(r_0):=\{z\in G_1:\ g_{G_1}(0,z)<r_0\}$ for some
$r_0\in(0,1)$.

We show that $g_{G_2}(0,F(z))=g_{G_1}(0,z)$ for $z\in G_1$. Fix a $z_0\in
G_1\setminus\{0\}$. Since $G_1$ is taut, there exists an extremal disc
$\phi\in\O(\D,G_1)$ for the pair $(0,z_0)$, i.e. $\phi(0)=0$, $\phi(\la_0)=z_0$,
 and $\lf_{G_1}^*(0,\phi(\la_0))=|\la_0|$ for some $\la_0\in(0,1)$. Thus
$g_{G_1}(0,\phi(\la_0))=|\la_0|$ and from Lemma \ref{l1} we obtain
$A_{G_1}(0;\phi'(0))=1$. From (5) we get
$A_{G_2}(F\circ\phi;(F\circ\phi)'(0))=1$. Thus once again from Lemma \ref{l1}
we get $g_{G_2}((F\circ\phi)(0),(F\circ\phi)(\la))=|\la|$ for any $\la\in\D$.
Hence $$g_{G_2}(0,F(z_0))=|\la_0|=g_{G_1}(0,z_0).$$Consequently,
$F^{-1}(B_{g_{G_2}}(r))\subset B_{g_{G_1}}(r)$ and
$F(B_{g_{G_1}}(r))\subset B_{g_{G_2}}(r)$, $r\in(0,1)$.

Now take $\{z_j\}\subset G_1$ such that $z_j\to\partial G_1$. Then
$g_{G_1}(0,z_j)\to 1$. Hence $g_{G_2}(0,F(z_j))\to 1$. Thus $F(z_j)\to
\partial G_2$. Therefore $F$ is a proper holomorphic map $G_1\to G_2$. In
particular, $F$ is surjective. Consequently, 
$F(B_{g_{G_1}}(r))=B_{g_{G_2}}(r)$ and
$F^{-1}(B_{g_{G_2}}(r))=B_{g_{G_1}}(r)$, $r\in(0,1)$.

Define $V:=\{z\in G_1:J_F(z)=0\}$. Then (see \cite{Rud}, Chapter 15.1)
\begin{itemize}
  \item $F(V)$ is an analytic subset of $G_2$;
  \item there exists an $m\in\mathbb{N}$ such that
  $\sharp F^{-1}(w)=m$ for $w\not\in F(V)$ and $\sharp
  F^{-1}(w)<m$ for $w\in F(V)$;
  \item $F:G_1\setminus F^{-1}(F(V))\to G_2\setminus F(V)$ is a holomorphic
  covering map.
\end{itemize}
Since $F:B_{g_{G_1}}(r_0)\to B_{g_{G_2}}(r_0)$ is biholomorphic, we conclude
that $m=1$. Thus $F$ is a biholomorphism.
\end{proof}

\end{document}